\documentclass[12pt]{article}
\usepackage{fullpage}
\usepackage{amssymb,amsmath,amsfonts,amsthm,latexsym,enumerate,url,cases}
\numberwithin{equation}{section}
\usepackage{hyperref}
\usepackage{indentfirst}
\hypersetup{colorlinks=true,citecolor=blue,linkcolor=blue,urlcolor=blue}

\newtheorem{theorem}{Theorem}[section] %
\newtheorem{lemma}[theorem]{Lemma} %
\newtheorem{corollary}[theorem]{Corollary} %

\usepackage{cite}

\begin{document}
\title{On the critical values of Burr's problem \footnote{e-mail addresse: yanxiaohui\b{ }1992@163.com(X.-H. Yan), wubingling911219@163.com(B.-L. Wu)}}

\author{  Xiao-Hui Yan\\
\small  School of Mathematics and Statistics, \\
\small  Anhui Normal University,  Wuhu  241003,  P. R. China\\
Bing-Ling Wu\\
\small  School of Science, \\
\small  Nanjing University of Posts and Telecommunications,  Nanjing  210023,  P. R. China\\
}
\date{}
\maketitle \baselineskip 18pt \maketitle \baselineskip 18pt

{\bf Abstract.} Let $A$ be a sequence of positive integers and $P(A)$ be the set of all integers which are the finite sum of distinct terms of $A$. For given positive integers $u\in\{4,7,8\}\cup\{u:u\ge11\}$ and $v\ge 3u+5$ we know that $u+v+1$ is the critical value of $b_3$ such that there exists a sequence $A$ of positive integers for which $P(A)=\mathbb{N}\backslash \{u<v<b_3<\cdots\}$. In this paper, we obtain the critical value of $b_k$.



\vskip 5mm

\section{Introduction}
Let $\mathbb{N}$ be the set of all nonnegative integers. For any sequence of positive integers $A=\{a_1<a_2<\cdots\}$, let $P(A)$ be the subset sum set of $A$, that is,
$$P(A)=\left\{\sum_{i}\varepsilon_i a_i:\sum_{i}\varepsilon_i<\infty, \varepsilon_i\in\{0,1\}\right\}.$$
Here we note that $0\in P(A)$.

In 1970, Burr \cite{Burr} posed the following problem: which sets $S$ of integers are equal to $P(A)$ for some $A$? For the existence of such set $S$ he mentioned that if the complement of $S$ grows sufficiently rapidly such as $b_1>x_0$ and $b_{n+1}\ge b_n^2$, then there exists a set $A$ such that $P(A)=\mathbb{N}\setminus\{b_1,b_2,\cdots\}$. But this result is unpublished. In 1996, Hegyv\'{a}ri \cite{Hegyvari} proved the following result.
\begin{theorem}\cite[Theorem 1]{Hegyvari} If $B=\{b_1<b_2<\cdots\}$ is a sequence of integers with $b_1\ge x_0$ and
$b_{n+1}\ge5b_n$ for all $n\ge1$, then there exists a sequence $A$ of positive integers for which $P(A)=\mathbb{N}\setminus B$.
\end{theorem}

In 2012, Chen and Fang \cite{ChenFang} obtained the following results.
\begin{theorem}\cite[Theorem 1]{ChenFang} Let $B=\{b_1<b_2<\cdots\}$ be a sequence of integers with $b_1\in\{4,7,8\}\cup\{b:b\ge11\}$ and
$b_{n+1}\ge3b_n+5$ for all $n\ge1$. Then there exists a sequence $A$ of positive integers for which $P(A)=\mathbb{N}\setminus B$.
\end{theorem}
\begin{theorem}\cite[Theorem 2]{ChenFang} \label{ChenFang2} Let $B=\{b_1<b_2<\cdots\}$ be a sequence of positive integers with $b_1\in\{3,5,6,9,10\}$ or $b_2=3b_1+4$ or $b_1=1$ and $b_2=9$ or $b_1=2$ and $b_2=15$. Then there is no sequence $A$ of positive integers for which $P(A)=\mathbb{N}\setminus B$.
\end{theorem}

Later, Chen and Wu \cite{ChenWu} further improved this result. By observing Chen and Fang's results we know that the critical value of $b_2$ is $3b_1+5$. In this paper, we study the problem of critical values of $b_k$. We call this problem critical values of Burr's problem.

In 2019, Fang and Fang \cite{FangFang2019} considered the critical value of $b_3$ and proved the following result.
\begin{theorem}\cite[Theorem 1.1]{FangFang2019} If $A$ and $B=\{1<b_1<b_2<\cdots\}$ are two infinite sequences of positive integers with $b_2=3b_1+5$ such that $P(A)=\mathbb{N}\setminus B$, then $b_3\ge4b_1+6$. Furthermore, there exist two infinite sequences of positive integers $A$ and $B=\{1<b_1<b_2<\cdots\}$ with $b_2=3b_1+5$ and $b_3=4b_1+6$ such that $P(A)=\mathbb{N}\setminus B$.
\end{theorem}

Recently, Fang and Fang \cite{FangFang2020} introduced the following definition. For given positive integers $b$ and $k\ge3$, define $c_{k}(b)$ successively as follows:

(i) let $c_k=c_{k}(b)$ be the least integer $r$ for which, there exist two infinite sets of positive integers $A$ and $B=\{b_1<b_2<\cdots<b_{k-1}<b_{k}<\cdots\}$ with $b_1=b$, $b_2=3b+5$ and $b_i=c_i(3\le i<k)$ and $b_k=r$ such that $P(A)=\mathbb{N}\setminus B$ and $a\le \sum_{\substack{a'<a\\a'\in A}}a'+1$ for all $a\in A$ with $a>b+1$;

(ii) if such $A,B$ do not exist, define $c_k=+\infty$.

In \cite{FangFang2020}, Fang and Fang proved the following result.
\begin{theorem}\cite[Theorem 1.1]{FangFang2020} For given positive integer $b\in\{1,2,4,7,8\}\cup\{b':b'\ge 11,b'\in\mathbb{N}\}$, we have
$$c_{2k-1}=(3b+6)(k-1)+b,~~c_{2k}=(3b+6)(k-1)+3b+5,~~k=1,2,\dots.$$
\end{theorem}

Naturally, we consider the problem that for any integer $b_1$ and $b_2\ge 3b_1+5$ if we can determine the critical value of $b_3$, instead of $b_2=3b_1+5$. This problem is posed by Fang and Fang in \cite{FangFang2019}. Recently, authors \cite{WuYan} answered this problem.

\begin{theorem}\cite{WuYan} \label{thm:1.6}If $A$ and $B=\{b_1<b_2<\cdots\}$ are two infinite sequences of positive integers with $b_2\ge 3b_1+5$ such that $P(A)=\mathbb{N}\setminus B$, then $b_3\ge b_2+b_1+1$.
\end{theorem}
\begin{theorem}\cite{WuYan} \label{thm:1.7}For any positive integers $b_1\in\{4,7,8\}\cup[11,+\infty)$ and $b_2\ge 3b_1+5$, there exists two infinite sequences of positive integers $A$ and $B=\{b_1<b_2<\cdots\}$ with $b_3=b_2+b_1+1$ such that $P(A)=\mathbb{N}\setminus B$.
\end{theorem}

In this paper, we go on to consider the critical value of $b_k$ for any integers $b_1$ and $b_2\ge 3b_1+5$. Motivated by the definition of Fang and Fang, we also introduce the following definition. For given positive integers $u$, $v\ge 3u+5$ and $k\ge3$, let $e_1=u$, $e_2=v$ and $e_k=e_k(u,v)$ be the least integer $r$ for which there exist two infinite sets of positive integer $A$ and $B=\{b_1<b_2<\cdots<b_{k-1}<b_k<\cdots\}$ with $b_i=e_i(1\le i<k)$ and $e_k=r$ such that $P(A)=\mathbb{N}\setminus B$ and $a\le \sum_{\substack{a'<a\\a'\in A}}a'+1$ for all $a\in A$ with $a>u+1$. If such sets $A,B$ do not exist, define $e_k=+\infty$.

In this paper, we obtain the following results.
\begin{theorem}\label{thm:1.1} For given positive integers $u\in\{4,7,8\}\cup\{u:u\ge11\}$, $v\ge 3u+5$, we have
\begin{equation}\label{eq:c}
e_{2k+1}=(v+1)k+u,~~e_{2k+2}=(v+1)k+v,~~k=0,1,\dots.
\end{equation}
\end{theorem}

\begin{corollary}\label{thm:1.2} For given positive integers $u\in\{4,7,8\}\cup\{u:u\ge11\}$, $v\ge 3u+5$ and $k\ge3$, we have
$$e_{k}=e_{k-1}+e_{k-2}-e_{k-3},$$
where $e_0=-1$, $e_1=u$, $e_2=v$.
\end{corollary}
If $u\in\{3,5,6,9,10\}$ or $u=1,v=9\ge 3u+5$ or $u=2, v=15\ge 3u+5$, by Theorem \ref{ChenFang2} we know that such sequence $A$ does not exist. So we only consider the case for $u\in\{4,7,8\}\cup\{u:u\ge11\}$. In fact, we find Corollary \ref{thm:1.2} first, but in the proof of Theorem \ref{thm:1.1} we follow Fang and Fang's method. Some of skills are similar to \cite{WuYan}. For the convenience of readers, we provide all the details of the proof.

\section{Proof of Theorem \ref{thm:1.1}}
For given positive integers $u\in\{4,7,8\}\cup\{u:u\ge11\}$ and $v\ge 3u+5$, we define
$$d_{2k+1}=(v+1)k+u,~~d_{2k+2}=(v+1)k+v,~~k=0,1,\dots.$$

\begin{lemma}\label{lem:2.1}
Given positive integers $u\in\{4,7,8\}\cup\{u:u\ge11\}$, $v\ge 3u+5$ and $k\ge3$. Then there exists
an infinite set $A$ of positive integers such that $P(A)=\mathbb{N}\setminus\{d_1,d_2,\dots\}$ and $a\le \sum_{\substack{a'<a\\a'\in A}}a'+1$ for all $a\in A$ with $a>u+1$.
\end{lemma}
\begin{proof} Let $s$ and $r$ be nonnegative integers with
$$v+1=(u+1)+(u+2)+\cdots+(u+s)+r,~~0\le r\le u+s.$$
Since $v\ge 3u+5$, it follows that $s\ge3$. Note that $u\ge4$. Then there exist integers $r_2,\dots,r_s$ such that
\begin{equation}\label{eq:2.1}
r=r_2+\cdots+r_s+\varepsilon(r),~~0\le r_2\le\cdots\le r_s\le u-1,
\end{equation}
where $\varepsilon(r)=0$ if $r=0$, otherwise $\varepsilon(r)=1$. If there is an index $3\le j\le s$ such that $r_j-r_{j-1}=u-1$, we replace $r_j$ and $r_{j-1}$ by $r_j-1$ and $r_{j-1}+1$. Then \eqref{eq:2.1} still holds and $r_j-r_{j-1}\le u-2$ for any index $3\le j\le s$.

We cite a result in \cite{ChenFang} that there exists a set of positive integers
$A_1$ with $A_1\subseteq[0,u-1]$ such that
$$P(A_1)=[0,u-1].$$
Let
$$a_1=u+1,~~a_s=u+s+r_s+\varepsilon(r),~~a_{t}=u+t+r_t,~~2\le t\le s-1.$$
Then
$$a_{t-1}<a_{t}\le a_{t-1}+u,~~2\le t\le s$$
and so
$$P(A_1\cup\{a_1,\dots,a_s\})=[0,a_{2}+\cdots+a_{s}+2u]\setminus\{u,a_{2}+\cdots+a_{s}+u\}.$$
Since
$$a_{2}+\cdots+a_{s}+u=(u+2+r_2)+\cdots+(u+s+r_s+\varepsilon(r))+u=v,$$
it follows that
\begin{equation}\label{eq:2.2}
P(A_1\cup\{a_1,\dots,a_s\})=[0,u+v]\setminus\{u,v\}.
\end{equation}
Let $a_{s+n}=(v+1)n$ for $n=1,2,\dots$. We will take induction on $k\ge1$ to prove that
\begin{equation}\label{eq:2.3}
P(A_1\cup\{a_1,\dots,a_{s+k}\})=[0,\sum_{i=1}^k a_{s+i} +u+v]\setminus\{d_1,d_2,\dots,d_{2m-1},d_{2m}\},
\end{equation}
where $m=k(k+1)/2+1$.

By \eqref{eq:2.2}, it is clear that
$$P(A_1\cup\{a_1,\dots,a_{s+1}\})=[0,a_{s+1}+u+v]\setminus\{d_1,d_2,d_3,d_4\},$$
which implies that \eqref{eq:2.3} holds for $k=1$.

Assume that \eqref{eq:2.3} holds for some $k-1\ge1$, that is,
\begin{equation}\label{eq:2.4}
P(A_1\cup\{a_1,\dots,a_{s+k-1}\})=[0,\sum_{i=1}^{k-1} a_{s+i} +u+v]\setminus\{d_1,d_2,\dots,d_{2m'-1},d_{2m'}\},
\end{equation}
where $m'=k(k-1)/2+1$. Then
\begin{eqnarray*}
a_{s+k}+P(A_1\cup\{a_1,\dots,a_{s+k}\})
=[(v+1)k,\sum_{i=1}^{k} a_{s+i} +u+v]\setminus D,
\end{eqnarray*}
where
$$D=\left\{d_{2k+1},d_{2k+2},\dots,d_{k(k+1)+1},d_{k(k+1)+2}\right\}.$$
Since $d_{2k+1}\le d_{k(k-1)+3}=d_{2m'+1}$, it follows that
\begin{equation*}
P(A_1\cup\{a_1,\dots,a_{s+k}\})=[0,\sum_{i=1}^k a_{s+i} +u+v]\setminus\{d_1,d_2,\dots,d_{2m-1},d_{2m}\},
\end{equation*}
where $m=k(k+1)/2+1$, which implies that \eqref{eq:2.3} holds.
Let $A=A_1\cup\{a_1,a_2,\dots\}$. We know that such $A$ satisfies Lemma \ref{lem:2.1}. This completes the proof of Lemma \ref{lem:2.1}.
\end{proof}

\begin{lemma}\label{lem:2.2}\cite[Lemma 1]{ChenFang}.
Let $A=\{a_1<a_2 <\cdots\}$ and $B=\{b_1<b_2 <\cdots\}$ be two sequences of positive integers with $b_1>1$ such that $P(A)=\mathbb{N}\backslash B$. Let $a_k<b_1<a_{k+1}$. Then
$$P(\{a_1,\cdots,a_i\})=[0,c_i], ~~i=1,2,\cdots,k,$$
where $c_1=1$, $c_2=3$, $c_{i+1}=c_i+a_i+1~(1\leq i\leq k-1)$, $c_k=b_1-1$ and $c_i+1\geq a_i+1~(1 \leq i \leq k-1)$.
\end{lemma}

\begin{lemma}\label{lem:2.3}
Given positive integers $u\in\{4,7,8\}\cup\{u:u\ge11\}$, $v\ge 3u+5$ and $k\ge3$. If $A$ is an infinite set of positive integers such that
$$P(A)=\mathbb{N}\setminus \{d_1<d_2<\cdots<d_{k-1}<b_{k}<\cdots\}$$
and $a\le \sum_{\substack{a'<a\\a'\in A}}a'+1$ for all $a\in A$ with $a>u+1$, then there exists a subset $A_1\subseteq A$ such that
$$P(A_1)=[0,d_1+d_2]\setminus\{d_1,d_2\}$$
and $\min \{A\setminus A_1\} >u+1$.
\end{lemma}
\begin{proof} Let $A=\{a_1<a_2<\cdots\}$ be an infinite set of positive integers such that
$$P(A)=\mathbb{N}\setminus \{d_1<d_2<\cdots<d_{k-1}<b_{k}<\cdots\}$$
and $a\le \sum_{\substack{a'<a\\a'\in A}}a'+1$ for all $a\in A$ with $a>u+1$. It follows from Lemma \ref{lem:2.2} that
$$P(\{a_1,\cdots,a_k\})=[0,u-1],$$
where $k$ is the index such that $a_k<u<a_{k+1}$. Since $v\ge 3u+5>u+1$, it follows that $u+1\in P(A)$. Hence, $a_{k+1}=u+1$. Then
$$P(\{a_1,\cdots,a_{k+1}\})=[0,2u]\setminus\{u\}.$$
Noting that $a_{k+t}>a_{k+1}=u+1$ for any $t\ge 2$, we have
$$a_{k+t}\le a_1+\cdots+a_{k+t-1}+1=a_{k+2}+\cdots+a_{k+t-1}+2u.$$
Then
$$P(\{a_1,\cdots,a_{k+2}\})=[0,a_{k+2}+2u]\setminus\{u,a_{k+2}+u\}.$$
If $a_{k+2}+\cdots+a_{k+t-1}+u\ge a_{k+t}$ and $a_{k+2}+\cdots+a_{k+t-1}\neq a_{k+t}$ for all integers $t\ge3$, then
$$P(\{a_1,\cdots,a_{k+t}\})=[0,a_{k+2}+\cdots+a_{k+t}+2u]\setminus\{u,a_{k+2}+\cdots+a_{k+t}+u\}.$$
Then $d_2\ge a_{k+2}+\cdots+a_{k+t}+u$ for any integer $t\ge3$, which is impossible since $d_2$ is a given integer. So there are some integers $3\le t_1<t_2<\cdots$ such that $a_{k+2}+\cdots+a_{k+t_i-1}+u< a_{k+t_i}$ or $a_{k+2}+\cdots+a_{k+t_i-1}= a_{k+t_i}$, and
$$P(\{a_1,\cdots,a_{k+t_1-1}\})=[0,a_{k+2}+\cdots+a_{k+t_1-1}+2u]\setminus\{u,a_{k+2}+\cdots+a_{k+t_1-1}+u\}.$$

If $a_{k+2}+\cdots+a_{k+t_1-1}+u< a_{k+t_1}$, then $d_2=a_{k+2}+\cdots+a_{k+t_1-1}+u$ and
$$P(\{a_1,\cdots,a_{k+t_1-1}\})=[0,d_1+d_2]\setminus\{d_1,d_2\},~~a_{k+t_1}>u+1.$$
So the proof is finished.

If $a_{k+2}+\cdots+a_{k+t_1-1}= a_{k+t_1}$, then
$$P(\{a_1,\cdots,a_{k+t_1}\})=[0,a_{k+2}+\cdots+a_{k+t_1}+2u]\setminus\{u,a_{k+t_1}+u,a_{k+2}+\cdots+a_{k+t_1}+u\}.$$
If $a_{k+t_1+1}>a_{k+t_1}+u$, then
$$d_2=a_{k+t_1}+u=a_{k+2}+\cdots+a_{k+t_1-1}+u$$
and
$$a_{k+2}+\cdots+a_{k+t_1-1}+2u=d_1+d_2.$$
Therefore,
$$P(\{a_1,\cdots,a_{k+t_1-1}\})=[0,d_1+d_2]\setminus\{d_1,d_2\},~~a_{k+t_1}>u+1.$$
So the proof is finished. If $a_{k+t_1+1}\le a_{k+t_1}+u$, then
$$P(\{a_1,\cdots,a_{k+t_1+1}\})=[0,a_{k+2}+\cdots+a_{k+t_1+1}+2u]\setminus\{u,a_{k+2}+\cdots+a_{k+t_1+1}+u\}.$$
By the definition of $t_2$ and $a_{k+t_1+1}\le a_{k+t_1}+u$ we know that $t_2\neq t_1+1$. Noting that $a_{k+2}+\cdots+a_{k+t-1}+u\ge a_{k+t}$ and $a_{k+2}+\cdots+a_{k+t-1}\neq a_{k+t}$ for any integer $t_1<t<t_2$, we have
$$P(\{a_1,\cdots,a_{k+t_2-1}\})=[0,a_{k+2}+\cdots+a_{k+t_2-1}+2u]\setminus\{u,a_{k+2}+\cdots+a_{k+t_2-1}+u\}.$$
Similar to the way to deal with $t_1$, we know that there is always a subset $A_1\subseteq A$ such that
$P(A_1)=[0,d_1+d_2]\setminus\{d_1,d_2\}$ and $\min\{A\setminus A_1\}>u+1$ or there exists an infinity sequence of positive integers $l_i\ge3$ such that
$$P(\{a_1,\cdots,a_{k+l_i}\})=[0,a_{k+2}+\cdots+a_{k+l_i}+2u]\setminus\{u,a_{k+2}+\cdots+a_{k+l_i}+u\}.$$
Since $d_2$ is a given integer, it follows that the second case is impossible. This completes the proof of Lemma \ref{lem:2.3}.
\end{proof}
\begin{lemma}\label{lem:2.4}
Given positive integers $u\in\{4,7,8\}\cup\{u:u\ge11\}$, $v\ge 3u+5$ and $k\ge3$. Let $A$ be an infinite set of positive integers such that
$$P(A)=\mathbb{N}\setminus \{d_1<d_2<\cdots<d_{k}<b_{k+1}<\cdots\}$$
and $a\le \sum_{\substack{a'<a\\a'\in A}}a'+1$ for all $a\in A$ with $a>u+1$ and let $A_1$ be a subset of $A$ such that
$$P(A_1)=[0,u+v]\setminus\{d_1,d_2\}$$
and $\min\{A\setminus A_1\}>u+1$. Write $A\setminus A_1=\{a_1<a_2<\cdots\}$. Then $v+1\mid a_i$ for $i=1,2,\dots, m$, and
$$P(A_1\cup\{a_1,\dots,a_m\})=[0,\sum_{i=1}^{m}a_i+u+v]\setminus\{d_1,d_2,\dots,d_{n}\},$$
where $m$ is the index such that
$$\sum_{i=1}^{m-1}a_i+v<d_k\le \sum_{i=1}^{m}a_i+v$$
and
$$d_{n}=\sum_{i=1}^m a_i+v.$$
\end{lemma}
\begin{proof} We will take induction on $k\ge3$. For $k=3$, by $a_1>u+1$ we know that
$$v<d_3=u+v+1\le a_1+v,$$
that is, $m=1$. It is enough to prove that $v+1\mid a_1$ and
$$P(A_1\cup\{a_1\})=[0,a_1+u+v]\setminus\{d_1,d_2,\dots,d_{n}\},$$
where $d_{n}=a_1+v$.
Since $d_3\notin P(A)$ and $[0,v-1]\setminus\{u\}\subseteq P(A_1)$ and
$$a_1\le \sum_{\substack{a'<a_1\\a'\in A}}a'+1=\sum_{a'\in A_1}a'+1=u+v+1=d_3< a_1+v,$$
it follows that $d_3=a_1+u$, that is, $a_1=v+1$. Since
$$P(A_1)=[0,u+v]\setminus\{d_1,d_2\},$$
it follows that
$$a_1+P(A_1)=[a_1,a_1+u+v]\setminus\{a_1+d_1,a_1+d_2\}.$$
Then
$$P(A_1\cup\{a_1\})=[0,a_1+u+v]\setminus\{d_1,d_2,d_3,d_4\},$$
where $d_4=a_1+v$.

Suppose that $ v+1\mid a_i$ for $i=1,2,\dots,m$ and
\begin{equation}\label{eq0}
P(A_1\cup\{a_1,\dots,a_{m}\})=[0,\sum_{i=1}^{m}a_i+u+v]\setminus\{d_1,d_2,\dots,d_{n}\},
\end{equation}
where $m$ is the index such that
$$\sum_{i=1}^{m-1}a_i+v<d_{k-1}\le \sum_{i=1}^{m}a_i+v$$
and
$$d_{n}=\sum_{i=1}^{m} a_i+v.$$
If $d_{k-1}< \sum_{i=1}^{m}a_i+v$, then $d_{k}\le \sum_{i=1}^{m}a_i+v$. Then the proof is finished. If $d_{k-1}=\sum_{i=1}^{m}a_i+v$, then $d_{k}=\sum_{i=1}^{m}a_i+v+u+1$. It follows that
$$\sum_{i=1}^{m}a_i+v<d_{k}\le \sum_{i=1}^{m+1}a_i+v.$$
It is enough to prove that $ v+1\mid a_{m+1}$ and
$$P(A_1\cup\{a_1,\dots,a_{m+1}\})=[0,\sum_{i=1}^{m+1}a_i+u+v]\setminus\{d_1,d_2,\dots,d_{n'}\},$$
where
$$d_{n'}=\sum_{i=1}^{m+1} a_i+v.$$
Since $a_{m+1}\neq d_k$ and
$$a_{m}<a_{m+1}\le \sum_{\substack{a<a_{m+1}\\a\in A}} a+1=\sum_{i=1}^{m}a_i+u+v+1=d_{k},$$
it follows that there exists a positive integer $T$ such that
$$a_{m+1}< (v+1)T+u\le a_{m+1}+v+1$$
and
$$(v+1)T+u\le d_k.$$
Note that $d_i=(v+1)T+u\notin P(A)$ for some $i\le k$ and $[1,v+1]\setminus\{u,v\}\in P(A_1)$. Hence, $(v+1)T+u=a_{m+1}+u$ or $(v+1)T+u=a_{m+1}+v$.
If $(v+1)T+u=a_{m+1}+u$, then $a_{m+1}=(v+1)T$. If $(v+1)T+u=a_{m+1}+v$, then $a_{m+1}=(v+1)(T-1)+u+1$.
Since $v \ge 3u+5$, it follows that
$$a_{m+1}+u<(v+1)(T-1)+v<a_{m+1}+v.$$
Note that $[u+1,v-1]\subseteq P(A_1)$. Then $(v+1)(T-1)+v\in P(A)$, which is impossible. Hence, $v+1\mid a_{m+1}$. Moreover, $a_{m+1}=(v+1)T$. Since
$$a_{m+1}+P(A_1\cup\{a_1,\dots,a_{m}\})=[a_{m+1},\sum_{i=1}^{m+1}a_i+u+v]\setminus\{a_{m+1}+d_1,\dots,a_{m+1}+d_{n}\}$$
and
$$a_{m+1}+d_1=(v+1)T+u\le d_k=\sum_{i=1}^{m}a_i+v+u+1=d_{n+1},$$
it follows from \eqref{eq0} that
$$P(A_1\cup\{a_1,\dots,a_{m+1}\})=[0,\sum_{i=1}^{m+1}a_i+u+v]\setminus\{d_1,d_2,\dots,d_{n'}\},$$
where
$$d_{n'}=a_{m+1}+d_n=\sum_{i=1}^{m+1} a_i+v.$$
This completes the proof of Lemma \ref{lem:2.4}.

\end{proof}

\begin{lemma}\label{lem:2.5}
Given positive integers $u\in\{4,7,8\}\cup\{u:u\ge11\}$, $v\ge 3u+5$ and $k\ge3$. If $A$ is an infinite set of positive integers such that
$$P(A)=\mathbb{N}\setminus \{d_1<d_2<\cdots<d_{k}<b_{k+1}<\cdots\}$$
and $a\le \sum_{\substack{a'<a\\a'\in A}}a'+1$ for all $a\in A$ with $a>u+1$, then $b_{k+1}\ge d_{k+1}$.
\end{lemma}
\begin{proof} By Lemma \ref{lem:2.3} we know that there exists $A_1\subseteq A$ such that
$$P(A_1)=[0,u+v]\setminus\{d_1,d_2\}$$
and $\min\{A\setminus A_1\}>u+1$. Write $A\setminus A_1=\{a_1<a_2<\cdots\}$. By Lemma \ref{lem:2.4} we know that $v+1\mid a_i$ for $i=1,2,\dots, m$ and
$$P(A_1\cup\{a_1,\dots,a_m\})=[0,\sum_{i=1}^{m}a_i+u+v]\setminus\{d_1,d_2,\dots,d_{n}\},$$
where $m$ is the index such that
$$\sum_{i=1}^{m-1}a_i+v<d_k\le \sum_{i=1}^{m}a_i+v$$
and
$$d_{n}=\sum_{i=1}^m a_i+v.$$
If $d_k<\sum_{i=1}^m a_i+v$, then $d_{k+1}\le\sum_{i=1}^m a_i+v=d_{n}$. Hence, $k+1\le n$. Thus, $b_{k+1}\ge d_{k+1}$.
If $d_k=\sum_{i=1}^m a_i+v$, then $d_{k+1}=\sum_{i=1}^m a_i+u+v+1$ and
\begin{equation}\label{eq:2.5}
P(A_1\cup\{a_1,\dots,a_m\})=[0,\sum_{i=1}^{m}a_i+u+v]\setminus\{d_1,\dots,d_{k}\}
\end{equation}
and
\begin{equation}\label{eq:2.6}
a_{m+1}+P(A_1\cup\{a_1,\dots,a_m\})=[a_{m+1},\sum_{i=1}^{m+1}a_i+u+v]\setminus\{a_{m+1}+d_1,\dots,a_{m+1}+d_{k}\}.
\end{equation}
Note that
$$a_{m}<a_{m+1}\le \sum_{\substack{a<a_{m+1}\\a\in A}} a+1=\sum_{i=1}^{m}a_i+u+v+1=d_{k+1}.$$
By $a_{m+1}\neq d_{k}$ we divide into two cases according to the value of $a_{m+1}$.

{\bf Case 1}: $d_{k}<a_{m+1}\le d_{k+1}$. It follows from \eqref{eq:2.5} and \eqref{eq:2.6} that
$$b_{k+1}\ge a_{m+1}+d_1\ge d_{k}+d_{1}+1=\sum_{i=1}^{m}a_i+v+u+1=d_{k+1}.$$

{\bf Case 2}: $a_{m}<a_{m+1}< d_{k}$. Similar to the proof of Lemma \ref{lem:2.4}, we know that there exists a positive integer $T$ such that
$$a_{m+1}=(v+1)T,~~a_{m+1}+d_1=(v+1)T+u\le d_k.$$
It follows from \eqref{eq:2.5} and \eqref{eq:2.6} that
\begin{equation}\label{eq:2.7}
P(A_1\cup\{a_1,\dots,a_{m+1}\})=[0,\sum_{i=1}^{m+1}a_i+u+v]\setminus\{d_1,\dots,d_{n'}\},
\end{equation}
where
$$d_{n'}=a_{m+1}+d_k.$$
Then $n'\ge k+1$. Thus $b_{k+1}\ge d_{k+1}$.
\end{proof}

\emph{Proof of Theorem \ref{thm:1.1}:} It follows from Theorem \ref{thm:1.6} and Theorem \ref{thm:1.7} that $e_3=(v+1)+u$. For $k\ge3$, suppose that $A$ is an infinite set of positive integers such that
$$P(A)=\mathbb{N}\setminus \{e_1<e_2<\cdots<e_{k}<b_{k+1}<\cdots\}$$
and $a\le \sum_{\substack{a'<a\\a'\in A}}a'+1$ for all $a\in A$ with $a>u+1$, where $e_i(1\le i\le k)$ is defined in \eqref{eq:c}. By Lemma \ref{lem:2.5} we have $b_{k+1}\ge d_{k+1}$. By Lemma \ref{lem:2.1} we know that $d_{k+1}$ is the critical value, that is, $e_{k+1}=d_{k+1}$. This completes the proof of Theorem \ref{thm:1.1}.

\noindent\textbf{Acknowledgments.} This work was supported by the National Natural Science Foundation of China, Grant No.11771211 and NUPTSF, Grant No.NY220092.

\renewcommand{\refname}{References}

\end{document}